\documentclass[
]{amsart}

\usepackage{tikz}
\usetikzlibrary{cd}

\usepackage{amsthm,amssymb,verbatim,amsmath,amsfonts, amscd, color} 
\usepackage{tikz-cd}
\usepackage{tikz}

\usepackage[left=25mm,right=25mm,top=30mm,bottom=30mm]{geometry}
\usepackage{enumitem}
\setlist{leftmargin=10mm, topsep=1mm, itemsep=1pt}

\usepackage[sc,outermarks]{titlesec}
\titleformat{\section}{\normalsize\sc\center}{\thesection}{0mm}{ }
\titlespacing{\section}{0mm}{3em}{1em}

\usepackage{hyperref} %

\theoremstyle{plain}\newtheorem{Theorem}{Theorem}[section]
\theoremstyle{plain}
\theoremstyle{plain}\newtheorem{Corollary}[Theorem]{Corollary}
\theoremstyle{plain}
\theoremstyle{plain}
\theoremstyle{definition}
\theoremstyle{definition}\newtheorem{Example}[Theorem]{Example}
\theoremstyle{definition}
\theoremstyle{definition}\newtheorem{Remark}[Theorem]{Remark}
\theoremstyle{definition}
\theoremstyle{definition}\newtheorem{Notation}[Theorem]{Notation}
\theoremstyle{definition}
\theoremstyle{definition}
\theoremstyle{definition}
\theoremstyle{definition}
\theoremstyle{definition}
\theoremstyle{definition}
\theoremstyle{definition}\newtheorem{Notation/Definition}[Theorem]{Notation/Definition}
\theoremstyle{definition}
\theoremstyle{plain}

\def\Syl{\mathrm{Syl}}

\def\dim{\mathrm{dim}}

  \def\Hom{\mathrm{Hom}}

 \def\Irr{\mathrm{Irr}}

\newcommand{\SL}{\operatorname{SL}}

\newcommand{\GL}{\operatorname{GL}}

\newtheoremstyle{defnew}{4ex}{}{}{}{\bf}{}{.5em}{}
\theoremstyle{defnew}

\theoremstyle{remark}

\theoremstyle{theorem}

\begin{document}

\title
{A reciprocity theorem of Robinson-Benson-Webb for finite-dimensional symmetric algebras}
\author{Shigeo Koshitani}
\address{S. Koshitani, Department of Mathematics and Informatics,
Graduate School of Science, Chiba
University, 1-33 Yayoi-cho, Inage-ku,Chiba 263-8522, Japan.}
\email{koshitan@math.s.chiba-u.ac.jp} 

\thanks{}



\begin{abstract} 
We generalize the reciprocity theorem of G.R.~Robinson, D. Benson and P. Webb 
between a finite group and its subgroup
to the case of finite-dimensional {\it symmetric} algebras over a field which are connected by
a bimodule for the two algebras.
\end{abstract}

\maketitle

\section{Introduction}
\noindent
In his paper \cite{Rob89}, G.R.~Robinson presents a reciprocity theorem for
simple/projective modules over the group algebras $kG$ and $kH$ where 
$k$ is a field and 
$H$ is a subgroup of a finite group $G$
(for the precise statement see Corollary~\ref{Robinson}).


Our main theorem stated below claims that Robinson's reciprocity theorem is extended for 
finite-dimensional $k$-algebras $A$ and $B$, 
and for an $(A,B)$-bimodule $M$ instead of $kG$ and $kH$,
provided $A$ and $B$ both are symmetric algebras and $M$ is projective as a left $A$-module
and also as a right $B$-module.
Actually the main theorem 
should work nicely when
we look for a kind of equivalence between the module categories of $A$-modules and $B$-modules
via the $(A,B)$-bimodule $M$. Now, our main result is:

\begin{Theorem}\label{MainTheorem}
Let $k$ be an algebraically closed field, and 
let $A$ and $B$ be finite-dimensional {\sf symmetric} $k$-algebras.
Further, let $M$ be an $(A,B)$-bimodule
such that $M$ is projective as a left $A$-module and also as a right $B$-module.
Then, for any simple right $A$-module $S$ and any simple right $B$-module $T$, 
$$
[P(T)\,|\,S\otimes_AM]^B=[P(S)\,|\,T\otimes_BM^*]^A
$$
where ${M^*}\!\!:=\Hom_k(M,k)$ is the $k$-dual of $M$, $P(S)$ and $P(T)$ are the
projective covers of $S$ and $T$, respectively, and $[P\,|\,U]^A$ for a projective
right $A$-module $P$ and a right $A$-module $U$ denotes the multiplicity of $P$ as a direct
summand of $U$.
\end{Theorem}

\noindent
Then, we immediately get original Robinson's reciprocity formula by applying 
Theorem~\ref{MainTheorem}
to the case that
$A:=kG$, $B:=kH$ and $M:=(kG){\downarrow}^{G\times G}_{G\times H}$
where the last term means the restriction of $kG$ as a $k(G\times G)$-module
to $k(G\times H)$-module.

\begin{Corollary}\label{Rob89}
[Theorem 3 in \cite{Rob89} and Lemma 5.2 in \cite{Gro02}]\label{Robinson}
Let $H$ be a subgroup of a finite group $G$, and let $S$ and $T$, respectively, be 
a simple $kG$-module and a simple $kH$-module. Then,
$
[P(T)\,|\,S{\downarrow}_H]^{kH}=[P(S)\,|\,T{\uparrow}^G]^{kG}$
where ${\downarrow}_H$ and ${\uparrow}^G$, respectively, mean the restriction to $H$
and the induction to $G$.
\end{Corollary}

\begin{Remark}
Footnote 1 on p.108 of \cite{Rob89} says that
D.~Benson and P.~Webb have known the same result as Corollary~\ref{Rob89}.
\end{Remark}

\noindent
Throughout this paper we use the following notation and convention.

\begin{Notation}
Here $k$ is an algebraically closed field, and $A$ and $B$ are finite-dimensional $k$-algebras.
All modules are finitely generated right modules unless stated otherwise.
We use the notation $M^*$, $P(U)$ and $[P\,|\,U]^A$ 
just as explained in Theorem~\ref{MainTheorem}.
Recall that $M^*$ becomes a $(B,A)$-bimodule by the action
$(bfa)(m):=f(amb)$ for $f\in M^*$, $a\in A$, $b\in B$ and $m\in M$. Further the $k$-algebra
$A$ is called {\sf symmetric} if $A\cong A^*$ as $(A,A)$-bimodules (see \cite[2.11]{Lin18}).
One of the important examples of symmetric $k$-algebras is the group algebra $kG$ of a
finite group $G$ (see \cite[Theorem 2.11.2]{Lin18}).
For an $A$-module $U$, we denote by $\Omega^0(U)$ the projective-free part of $U$,
namely, $U=\Omega^0(U)\oplus P$ such that
$P$ is a projective $A$-module and $\Omega^0(U)$ is an $A$-module satisfying that 
any indecomposable direct summand of $\Omega^0(U)$ is non-projective.
We write $\Irr_k(A)$ for the set of all non-isomorphic simple $A$-modules.
For $A$-modules $U$ and $V$, we denote by ${\underline{\Hom}}_A(U,V)$ the $k$-vector space
$\Hom_A(U,V)/\Hom_A^{\mathrm{pr}}(U,V)$ where $\Hom_A^{\mathrm{pr}}(U,V)$ is the set of all
elements $f\in\Hom_A(U,V)$ such that $f$ factor through projective $A$-modules
(see \cite[4.13]{Lin18}).
For a $kG$-module $X$ and a $kH$-module $Y$ we denote by
$X{\downarrow}_H$ and $Y{\uparrow}^G$ as in Corollarly~\ref{Robinson}.
Sometimes we write also like $X{\uparrow}^G_H$ and $Y{\downarrow}_H^G$ 
to mean the same but in order to emphasize
the original groups.
For the other notion and terminology, consult the book \cite{Lin18}. 
\end{Notation}

\section{The main theorem}
\begin{proof}[Proof of Theorem~\ref{MainTheorem}]
We can write
\begin{equation*}
T\otimes_BM^*=\Omega^0(T\otimes_BM^*)\oplus(\oplus_{S'\in\Irr_k(A)}\,m_{S'}\times P(S'))
\text{ for integers }m_{S'}\geq 0
\end{equation*}
and
\begin{equation*}
S\otimes_AM=\Omega^0(S\otimes_AM)\oplus(\oplus_{T'\in\Irr_k(B)}\,n_{T'}\times P(T'))
\text{ for integers }n_{T'}\geq 0.
\end{equation*}
These imply that
\begin{equation}\label{m_S}
\dim_k\,\Hom_A(S,T\otimes_BM^*)=\dim_k\,\Hom_A(S,\Omega^0(T\otimes_BM^*))+m_S
\end{equation}
and
\begin{equation}\label{n_T}
\dim_k\,\Hom_B(S\otimes_AM,T)=\dim_k\,\Hom_B(\Omega^0(S\otimes_AM),T)+n_T\,.
\end{equation}
Now, it follows that as $k$-spaces
\begin{align*}
\Hom_A(S,\Omega^0(T\otimes_BM^*))&
\cong\underline{\Hom}_A(S,\Omega^0(T\otimes_BM^*))\text{ by \cite[Corollary 4.13.4]{Lin18} }
\\&
=\underline{\Hom}_A(S,T\otimes_BM^*))\text{ by the definitions of }\Omega^0\text{ and }
\underline{\Hom}
\\&
\cong\underline{\Hom}_B(S\otimes_AM,T) \text{ by \cite[Proposition 2.15.5]{Lin18}  }
\\&
\cong\Hom_B(\Omega^0(S\otimes_AM),T) \text{\ \  just as above,}
\end{align*}
namely,
\begin{equation}\label{adjoint3}\quad
\dim_k\,\Hom_A(S,\Omega^0(T\otimes_BM^*))=\dim_k\,\Hom_B(\Omega^0(S\otimes_AM),T).
\end{equation}
On the other hand, \cite[Theorem 2.12.7]{Lin18} implies that
\begin{equation}\label{adjoint4}
\dim_k\,\Hom_A(S,T\otimes_BM^*)=\dim_k\,\Hom_B(S\otimes_AM,T).
\end{equation}
Hence
\begin{align*}
m_S&=\dim_k\,\Hom_A(S,T\otimes_BM^*)-\dim_k\,\Hom_A(S,\Omega^0(T\otimes_BM^*))
\text{ \ \ by (\ref{m_S})}
\\&
=\dim_k\,\Hom_B(S\otimes_AM,T)-\dim_k\,\Hom_B(\Omega^0(S\otimes_AM),T)
\text{ \ \ by (\ref{adjoint4}) and (\ref{adjoint3})}
\\&=
n_T\text{ \ \ by (\ref{n_T})}.
\end{align*}
This completes the proof.
\end{proof}

\begin{Remark}
We hope Theorem~\ref{MainTheorem} might be applied even for wider algebras
such as Iwahori-Hecke algebras (see Chap.7 of \cite{GP00} and  \S2 of \cite{Miy25}).
\end{Remark}

\section{Example}
\noindent
Finally we give an example where Corollary~\ref{Robinson}
does not work while our main result Theorem~\ref{MainTheorem} does work.
\begin{Example}
Let $k$ be an algebraically closed field of $3$, and set $G:={\sf M}_{12}$, the Mathieu group of 
degree $12$ and $H:=\SL_3(3)$. It follows from 
\cite[\S 9, (8) in the table on p.2050 and Proposition 59]{NU09} the following.
That is, $G$ and $H$ have Sylow $3$-subgroups which are isomorphic to $3_+^{1+2}$, the
extra-special group of order $27$ of exponent $3$, so we identify them, namely
let $P\in\Syl_3(G)\cap\Syl_3(H)$ with $P\cong 3_+^{1+2}$,
Then $P$ has a normal subgroup $Q$ of index $3$ such that 
$N_G(Q)\cong N_H(Q)\cong Q\rtimes\GL_2(3)$ (a semi-direct product). 
Actually we can identify them, so 
set $N:=N_G(Q)=N_H(Q)$ and then $N\gneqq N_G(P)=N_H(P)\cong P\rtimes(C_2\times C_2)$
where $C_2$ is the cyclic group of order $2$. 
We can set $\Irr_k(kN):=\{k_N, 1=1_{kN}, 2, 2^*, 3_1, 3_2\}$ where the numbers mean the $k$-dimensions
(see \cite[(7.4)]{KW99} and \cite[p.1223]{Kos87}). Note that the group $H$ in \cite{Kos87} is the
same as $N$ in \cite{KW99}, which is actually the $N$ in here as well.
Now, set 
$$M:=kG\otimes_{kN}kH, \ A:=kG\text{ and }B:=kH.$$
Thus, $M$ is an $(A,B)$-bimodule.
It is easy to know by the Mackey decomposition theorem that 
$M$ is projective (free, in fact) as a left $A$-module, and so is
as a right $B$-module.
Now, we use the same notation in \cite[\S 0]{KW99}.
Then Theorem (ii) in [ibid] says that there is a simple $A$-module $S:=45'$ of $k$-dimension $45$
such that 
\begin{equation}\label{45'}
S{\downarrow}^G_N=W\oplus P(3_2)
\end{equation}
for an indecomposable $kN$-module $W:=f(S):=f_{(G,Q,N)}(S)$ with vertex $Q$
and $3_2$ is a simple $kN$-module of $k$-dimension $3$ (see (7.4) in [ibid]
and see also \cite[Chap.4 \S 4.1]{NT88} for the notation $f_{(G,Q.N)}$ 
which is the Green correspondence).
Then, by the information in \cite{Kos87}, we know that
\begin{equation}\label{3_2}
\mathcal P:=P(3_2){\uparrow}_N^H=P(15)\oplus P(15^*)\oplus P(7)\oplus P(27)
\end{equation}
where $7, 27\in\Irr_k(B)$ (see \cite{Kos87}). So that
\begin{equation}\label{S-tensor-M}
S\otimes_AM=S{\downarrow}^G_N{\uparrow}_N^H=W{\uparrow}_N^H\oplus\mathcal P.
\end{equation}
Let $f':=f_{(H,Q,N)}$ be the Green correspondence with respect to $(H,Q,N)$. 
Since we know that $Q\cap Q^h=1$ for any $h\in H-N$, 
it holds that 
\begin{equation}\label{W}
W{\uparrow}_N^H=X\oplus\text{(proj)}
\end{equation}
where $X:={f'}^{-1}\circ f(S)={f'}^{-1}(W)$.
Further we know that the projective part of (\ref{W}) is $P(T)\oplus P(T^*)$,
where $T:=15\in\Irr_k(B)$ (see \cite{Kos87}).  Thus, by (\ref{3_2}), (\ref{S-tensor-M}) and (\ref{W}),
it holds that
\begin{equation}\label{S-tensor-M_2}
S\otimes_AM=X\oplus\Big(2\times P(T)\oplus 2\times P(T^*)\oplus P(7)\oplus P(27)\Big)
\end{equation}
Thus, 
by making use of Theorem~\ref{MainTheorem} that
\begin{equation}\label{claim}
2=[P(T)\,|\,S\otimes_AM]^B=[P(S)\,|\,T\otimes_BM^*]^A.
\end{equation}
On the other hand, it follows from \cite[(1.3)Proposition (v)]{Kos87} that
\begin{equation}\label{SellingPoint}
T\otimes_BM^*=T{\downarrow}^H_N{\uparrow}_N^G \supsetneqq 3_2{\uparrow}_N^G
=P(S)\oplus U
\end{equation}
for an $A$-module $U$ with $P(S)\,{\not|}\,U$.
Perhaps the readers would agree that it should be hard to know that
$T{\downarrow}^H_N{\uparrow}_N^G$ has $2\times P(S)$ as direct summands. 
And these kind of arguments could be necessary sometimes when one wants to know
the relation-ships between representations of $A$ and $B$.
Hopefully our main result Theorem~\ref{MainTheorem} 
could be useful for the other purpose.
\end{Example}

\noindent {\bf Acknowledgement.} 
{\small
The author is grateful to Markus Linckelmann for his advice.}


\begin{thebibliography}{9999999} 


\bibitem[GP00]{GP00}
M. Geck, G. Pfeiffer.
\textit{Characters of Finite Coxeter Groups and Iwahori-Hecke Algebras.} 
Clarendon Press, 2000.

\bibitem[Gro02]{Gro02} J. Grodal.
Higher limits via subgroup complexes.
Ann. of Math. {\bf 155} (2002), 405--457.

\bibitem[Kos87]{Kos87}
S. Koshitani.
The Loewy structure of the projective indecomposable modules for $\SL(3,3)$ and
its automorphism group in characteristic $3$.
Comm.~Algebra {\bf 15} (1987), 1215--1253.

\bibitem[KW99]{KW99}
S. Koshitani, K. Waki.
The Green correspondents of the Mathieu group $M_{12}$
in characteristic $3$.
Comm.~Algebra {\bf 27} (1999), 37--66.


\bibitem[Lin18]{Lin18} M.~Linckelmann. 
The Block Theory of Finite Group Algebras, vol.1. 
Cambridge Univ.~Press, Cambridge, 2018.

\bibitem[Miy25]{Miy25}
H. Miyachi.
\textit{Two reciprocities on Hecke algebras Part 1: Robinson's reciprocity.}
In: Cohomology Theory of Finite Groups and Related Topics,
RIMS K\^oky\^uroku  
\textbf{2306}, Kyoto University, 4 pages, 2025.


\bibitem[NT88]{NT88}
H.~Nagao, Y. Tsushima.
Representations of Finite Groups, Academic Press, New York, 1988.

\bibitem[NU09]{NU09}
R. Narasaki, K. Uno.
Isometries and extra special Sylow groups of order $p^3$.
J.~Algebra {\bf 322} (2009), 2027--2068.



\bibitem[Rob89]{Rob89}
G.R.~Robinson.
On projective summands of induced modules.
J.~Algebra {\bf 122} (1989), 106--111.
\end{thebibliography}
\end{document}